\documentclass{amsart}

\usepackage[utf8]{inputenc}
\usepackage{verbatim}
\usepackage{graphicx}
\usepackage{graphicx,caption2,psfrag,float,color}
\usepackage{amssymb}
\usepackage{amscd}
\usepackage{amsmath}

\newtheorem{theorem}{Theorem}[section]

\newtheorem{lemma}[theorem]{Lemma}

\numberwithin{equation}{subsection}
\newtheorem{definition}[theorem]{Definition}

\title{Large gaps between consecutive prime numbers containing square-free numbers and perfect powers of prime numbers}
\author{Helmut Maier and Michael Th. Rassias}
\date{\today}
\address{Department of Mathematics, University of Ulm, Helmholtzstrasse 18, 8901 Ulm, Germany.}
\email{helmut.maier@uni-ulm.de}
\address{Department of Mathematics, ETH-Z\"{u}rich, R\"{a}mistrasse 101, 8092 Z\"{u}rich, Switzerland \& Department of Mathematics, Princeton University, Fine Hall, Washington Road, Princeton, NJ 08544-1000, USA}
\email{michail.rassias@math.ethz.ch, michailrassias@math.princeton.edu}
\thanks{}

\begin{document}

 \maketitle
 
\begin{abstract} We prove a modification as well as an improvement of a result of K. Ford, D. R. Heath-Brown and S. Konyagin \cite{konyagin} concerning prime avoidance of square-free
numbers and perfect powers of prime numbers. 


\end{abstract}

\section{Introduction}
In their paper \cite{konyagin}, K. Ford, D. R. Heath-Brown and S. Konyagin prove the 
existence of infinitely many ``prime-avoiding" perfect $k$-th powers for any positive integer $k$.\\
They give the following definition of prime avoidance: an integer $m$ is called prime avoiding with constant $c$, if $m+u$ is composite for all integers $u$ satisfying\footnote{We denote by
$\log_2 x=\log\log x$, $\log_3 x=\log\log\log x$, and so on.}
$$|u|\leq c\:\frac{\log m\log_2m\log_4m}{(\log_3m)^2}.$$
In this paper, we prove the following two theorems:
\begin{theorem}\label{firstmain}
There is a constant $c>0$ such that there are infinitely many prime-avoiding square-free 
numbers with constant $c$.
\end{theorem}
\begin{theorem}\label{main}
For any positive integer $k$, there are a constant $c=c(k)>0$ and infinitely many perfect $k$-th powers of prime numbers which are prime-avoiding with constant $c$.
\end{theorem}
\section{Proof of the Theorem \ref{firstmain}}
We largely follow the proof of \cite{konyagin}.
\begin{lemma}\label{x:lemk1}
For large $x$ and $z\leq x^{\log_3 x/(10\log_2 x)}$, we have
$$|\{n\leq x\::\: P^+(n)\leq z \}|\ll \frac{x}{(\log x)^5},$$
where $P^+(n)$ denotes the largest prime factor of a positive integer $n$.
\end{lemma}
\begin{proof} This is Lemma 2.1 of \cite{konyagin} (see also \cite{rankin}).
\end{proof}
\begin{lemma}\label{x:lemk2}
Let $\mathcal{R}$ denote any set of primes and let $a\in \{-1,1\}$.  Then, for large $x$, we have
\[
 |\{ p\le x\: :\: p\not \equiv a(\bmod{r}) \; (\forall r\in \mathcal{R}) \}|
\ll \frac{x}{\log x} \prod_{\substack{p\in \mathcal{R}\\ p\le x}} \left(1-\frac{1}{p}\right).
\] 
\end{lemma}
\textit{Note.} Here and in the sequel $p$ will always denote a prime number.
\begin{proof} This is Lemma 2.2 of \cite{konyagin} (see also \cite{halb}).
\end{proof}
\begin{lemma}\label{x:lemak3}
It holds
$$\prod_{p\leq w}\left(1-\frac{1}{p} \right)=\frac{e^{-\gamma}}{\log w}\left( 1+O\left(\frac{1}{\log w}\right)\right),\ \ (w\rightarrow+\infty)\:,$$
where $\gamma$ denotes the Euler-Mascheroni constant.
\end{lemma}
\begin{proof}
This is well known (cf. \cite{hardy}, p. 351).
\end{proof}
\begin{definition}\label{x:def24}
Let $x$ be a sufficiently large number. Let also $c_1$ and $c_2$ be two positive constants,
to be chosen later and set 
$$z=x^{c_1\log_3x/\log_2x},\: y=c_2\frac{x\log x \log_3 x}{(\log_2 x)^2}.$$
Let 
\begin{align*}
P_1&=\left\{p\::\: p\leq \log x\ \text{or}\ z<p\leq \frac{x}{4}\right\},\\
P_2&=\left\{p\::\: \log x< p\leq z\right\},\\
U_1&=\{ u\in[-y,y]\::\: u\in\mathbb{Z},\: p\mid u\ \text{for at least one}\ p\in P_1\},\\
U_2&=\{u\in[-y,y]\::\: u\not\in U_1,\: u\not\in\{-1,0,1\}\},\\
U_3&=\{u\in U_2\::\: |u|\ \text{is a prime number}\},\\
U_4&=\{u\in U_2\::\: |u|\ \text{is composed only of prime numbers}\ p\in P_2\}.
\end{align*}
\end{definition}
\begin{lemma}\label{x:paliolem25}
We have
$$U_2=U_3\cup U_4\:.$$
\end{lemma}
\begin{proof}
Assume that $u\in U_2\setminus U_4$. Then, by Definition \ref{x:def24}, there is a prime number
$p_0\not\in P_2$ with $p_0\mid |u|$. Since by Definition \ref{x:def24} we know that
$u\not\in U_1$, we have
$$p_0> \frac{x}{4}.$$
Let $p_1$ be a prime with $p_1\mid \frac{|u|}{p_0}.$ Then
$$|u|\geq p_0p_1> \frac{x}{4}\log x>y.$$
Thus, $p_1$ does not exist and we have $|u|=p_0$ and therefore $u\in U_3$.
\end{proof}
\begin{lemma}\label{lemma26}
We have 
$$|U_4|\ll\frac{x}{(\log x)^4}\:.$$
\end{lemma}
\begin{proof}
This follows immediately from Lemma \ref{x:lemk1}.
\end{proof}
\begin{definition}
Let
$$U_5=\{u\in U_3\::\: p\nmid u+1\ \text{for all}\ p\in P_2\}.$$
\end{definition}
\begin{lemma}\label{lemma28}
We can choose the constants $c_1$, $c_2$, such that
$$|U_5|\leq \frac{x}{3\log x}\:.$$
\end{lemma}
\begin{proof}
By Lemma \ref{x:lemk2}, we have
$$|U_5|\ll\frac{y}{\log y}\prod_{p\in P_2}\left(1-\frac{1}{p}\right).$$
Additionally, by Lemma \ref{x:lemak3}, we have
$$\prod_{p\in P_2}\left(1-\frac{1}{p}\right)=\frac{(\log_2 x)^2}{c_1\log x\cdot \log_3 x}\left(1+O\left(\frac{1}{\log_2 x}\right)\right).$$
Therefore,
$$|U_5|\ll \frac{c_2}{c_1}\frac{x}{\log x},$$
which proves Lemma \ref{lemma28}
\end{proof}
\begin{definition}\label{def29}
We set 
$$U_6=U_4\cup U_5\cup \{-1,0,1\}\:.$$
\end{definition}
\begin{lemma}
We have 
$$|U_6|\leq \frac{x}{2\log x}\:.$$
\end{lemma}
\begin{proof}
This follows from Definition \ref{def29} and Lemmas \ref{lemma26}, \ref{lemma28}.
\end{proof}
\begin{definition}\label{def211}
Let 
$$P_3=\left\{p\::\: \frac{x}{4}<p\leq x\right\}.$$
Let $\Phi\::\: U_6\rightarrow P_3$ be an injective map. Such a map $\Phi$ exists, since $$|P_3|\geq |U_6|.$$ We denote
$$\Phi(u)=p_u.$$
We define 
$$N=\prod_{p\leq\frac{x}{4}}p\:\prod_{u\in U_6}p_u.$$
We determine $m_0$ by the inequalities
$$1\leq m_0\leq N$$
and by the congruences: 
\begin{align*}
(1)\ \ \ m_0&\equiv 0\: (\bmod\:p),\ (p\in P_1)\\
(2)\ \ \ m_0&\equiv 1\: (\bmod\:p),\ (p\in P_2)\\
(3)\ \ \ m_0&\equiv -u\: (\bmod\:p_u),\ (p\in \Phi(U_6))\:.
\end{align*}
\end{definition}
\begin{lemma}\label{lemma212}
Let $m\geq 2y$, $m\equiv m_0\: (\bmod\:N)$. Then $m+u$ is composite for $u\in[-y,y]$.
\end{lemma}
\begin{proof}
If $u\in U_6$, then by the congruence (3) of Definition \ref{def211}, we have
$$p_u\mid m_0+u\:.$$
For $u\not\in U_6$, by the definition of the sets $U_1,\dots, U_5$, there is a $p\in P_1$,
such that $p\mid u$ or there is a $p\in P_2$, such that $p\mid u+1$. In both cases $p\mid m+u$, due to the congruences (1) and (2).\\ 
Thus, for all $u\in[-y,y]$ there is a prime $p$ with $p\mid m+u$ and $p<m+u$. Hence,
$m+u$ is composite for all $u\in[-y,y].$
\end{proof}
\textit{Proof of Theorem \ref{firstmain}.}
We now consider the arithmetic progression
\[
m=kN+m_0,\ k\in\mathbb{N}\:.\tag{*}
\]
By elementary methods (see Heath-Brown \cite{heath} for references) the arithmetic progression (*) contains a square-free number 
\[
m\leq N^{3/2+\varepsilon},\tag{1}
\]
where $\varepsilon>0$ is arbitrarily small.\\
By the prime number theorem, we have
\[
N\leq e^{x+o(x)}.\tag{2}
\]
By Lemma \ref{lemma212}, we know that $m+u$ is a composite number for 
$u\in[-y,y].$ By the estimates (1) and (2), we obtain
$$y\geq c\:\frac{\log m\log_2 m\log_4 m}{(\log_3 m)^2}$$
for a constant $c>0$, which proves Theorem \ref{firstmain}.
\section{Sieve estimates}
We introduce some notations borrowed with minor modifications form \cite{konyagin}.\\
Let 
$$\mathcal{A}=\text{a finite set of integers}$$
$$\mathcal{P}=\text{a subset of the set of all prime numbers}\:.$$
For each prime $p\in\mathcal{P}$, suppose that we are given a subset $\mathcal{A}_p\subseteq \mathcal{A}.$\\ Let $\mathcal{A}_1=\mathcal{A}$, 
$$P(z)=\prod_{\substack{p<z \\ p\in\mathcal{P}}}p$$
and
$$S(\mathcal{A},\mathcal{P},z)=\left| \mathcal{A}\setminus \bigcup_{p\mid P(z)}\mathcal{A}_p \right|\:.$$
Then for a positive square-free integer $d$ composed of primes of $\mathcal{P}$
we define:
$$\mathcal{A}_d=\bigcap_{p\mid d} \mathcal{A}_p\:.$$
We assume that there is a multiplicative function $\omega(\cdot)$, such that for any $d$ as above
$$|\mathcal{A}_d|=\frac{\omega(d)}{d}\: X+R_d\:,$$
for some $R_d$, where $X=|\mathcal{A}|$.\\
We set
$$W(z)=\prod_{p\mid P(z)}\left(1-\frac{\omega(p)}{p}\right)\:.$$
\begin{lemma}\textsc{(Brun's sieve)}\label{x:brun}\\
Let the notations be as above. Suppose that:\\
1. $|R_d|\leq \omega(d)$ for any square-free integer $d$ composed of primes of 
$\mathcal{P}$\\
2. there exists a constant $A_1\geq 1$, such that
$$0\leq \frac{\omega(p)}{p}\leq 1-\frac{1}{A_1}\: $$\\
3. there exist constants $\kappa>0$ and $A_2\geq 1$, such that
$$\sum_{w\leq p<z}\frac{\omega(p)\log p}{p} \leq \kappa\log\frac{z}{w}+A_2\:,\ \text{if}\ \ 2\leq w\leq z\:.$$
Let $b$ be a positive integer and let $\lambda$ be a real number satisfying
$$0<\lambda e^{1+\lambda}<1\:.$$
Then
$$S(\mathcal{A},\mathcal{P},z)\leq XW(z)\left\{ 1+2\:\frac{\lambda^{2b+1}e^{2\lambda}}{1-\lambda^2e^{2+2\lambda}}\:\exp\left((2b+3)\frac{c_1}{\lambda \log z}  \right) \right\}+O\left(z^{2b+\frac{2.01}{e^{2\lambda/\kappa}-1}} \right)\:.$$
\end{lemma}
\begin{proof}
This is part of Theorem 6.2.5 of \cite{cojo}.
\end{proof}
\section{Primes in arithmetic progressions}
The following definition is borrowed from \cite{maier}.
\begin{definition}\label{x:definition31}
Let us call an integer $q>1$ a ``good" modulus, if $L(s,\chi)\neq 0$ for all characters $\chi\bmod\:q$ and all $s=\sigma+it$ with
$$\sigma>1-\frac{C_1}{\log\left[ q(|t|+1)\right]}.$$
This definition depends on the size of $C_1>0$.
\end{definition}
\begin{lemma}\label{x:lem32}
There is a constant $C_1>0$ such that, in terms of $C_1$, there exist arbitrarily large values of $x$ for which the modulus
$$P(x)=\prod_{p<x}p$$ is good.
\end{lemma}
\begin{proof}
This is Lemma 1 of \cite{maier}
\end{proof}
\begin{lemma}
Let $q$ be a good modulus. Then 
$$\pi(x;q,a)\gg \frac{x}{\phi(q)\log x}\:,$$
uniformly for $(a,q)=1$ and $x\geq q^D$.\\
Here the constant $D$ depends only on the value of $C_1$ in Lemma \ref{x:lem32}.
\end{lemma}
\begin{proof}
This result, which is due to Gallagher \cite{gala}, is Lemma 2 from \cite{maier}.
\end{proof}
\section{Congruence conditions for the prime-avoiding number}
Let $x$ be a large positive number and $y$, $z$ be defined as in Definition \ref{x:def24}.\\
Set 
$$P(x)=\prod_{p\leq x}p\:.$$
We will give a system of congruences that has a single solution $m_0$, with $$0\leq m_0\le P(x)-1$$ having the property that the interval $[m_0^k-y,m_0^k+y]$ 
contains only few prime numbers.
\begin{definition}
We set
\begin{align*}
\mathcal{P}_1&=\{p\::\: p\leq \log x\ \ \text{or}\ \ z<p\leq x/40k\}\:,\\
\mathcal{P}_2&=\{p\::\: \log x<p\leq z\}\:,\\
\mathcal{U}_1&=\{u\in[-y,y],\:u\in\mathbb{Z},\: p\mid u\ \text{for at least one}\: 
p\in\mathcal{P}_1\}\:,\\
\mathcal{U}_2&= \{u\in[-y,y]\::\:u\not\in  \mathcal{U}_1\}\:,\\
\mathcal{U}_3&= \{u\in[-y,y]\::\:|u|\ \text{is prime}\}\:,\\
\mathcal{U}_4&= \{u\in[-y,y]\::\: P^+(|u|)\leq z\}\:,\\
\mathcal{U}_5&= \{u\in \mathcal{U}_3\:: p\nmid u+2^k-1\ \text{for}\ p\in \mathcal{P}_2\}
\end{align*}
\end{definition}
\begin{lemma}
We have $$\mathcal{U}_2=\mathcal{U}_3\cup\mathcal{U}_4\:.$$
\end{lemma}
\begin{proof}
This is Lemma \ref{x:paliolem25}.
\end{proof}
\begin{lemma}\label{x:lem43}
We have 
$$|\mathcal{U}_4|\ll\frac{x}{(\log x)^4}\:.$$
\end{lemma}
\begin{proof}
This is Lemma \ref{lemma26}.
\end{proof}
\begin{lemma}\label{x:lem44}
We can choose the constants $c_1$, $c_2$ such that
$$|\mathcal{U}_5|\leq \frac{x}{30k\log x}\:.$$
\end{lemma}
\begin{proof}
We have 
$$\mathcal{U}_5=\mathcal{U}_{5,1}\cup(-\mathcal{U}_{5,2})$$
with
$$\mathcal{U}_{5,1}=\{u\in \mathcal{U}_3\::\: p\nmid u+2^k-1\ \text{for}\ p\in\mathcal{P}_2\}$$  
\begin{align*}
\mathcal{U}_{5,2}&=\{u\in \mathcal{U}_3\::\: p\nmid -u+2^k-1\ \text{for}\ p\in\mathcal{P}_2\}\\
&=\{u\in \mathcal{U}_3\::\: p\nmid u-2^k+1\ \text{for}\ p\in\mathcal{P}_2\}\:.
\end{align*}
We only give details for the estimate of $|\mathcal{U}_{5,1}|$, since the estimate of 
$|\mathcal{U}_{5,2}|$ is completely analogous.\\
We apply Lemma \ref{x:brun} with 
$$\mathcal{A}=\{n\::\:n\leq y\}\:.$$
For $p\in \mathcal{P}_1$ we define $\mathcal{A}_p$ by 
$$\mathcal{A}_p=\{n\in\mathcal{A}\::\: n\equiv 0\:(\bmod\: p)\ \text{or}\ n\equiv 1-2^k\:(\bmod\: p) \}\:.$$
We check whether the conditions for the application of Lemma \ref{x:brun} are satisfied. For $d\mid P(y)$ we set:
$$\mathcal{A}_d=\bigcap_{p\mid d}\mathcal{A}_p\:.$$
We partition the interval $(0,y]$ into $\lfloor y/d\rfloor$ subintervals of length $d$ and possibly one additional interval $I_{last}$ of length less than $d$.\\
Let $\omega(d)$ be the number of the solutions $(\bmod\: d)$ of the system
\begin{align*}
n&\equiv0\:(\bmod\: p),\ \ p\in\mathcal{P}_1\cup\mathcal{P}_2\\
n&\equiv 1-2^k\:(\bmod\: p),\ \ p\in\mathcal{P}_2 \:.\tag{**}
\end{align*}
By the Chinese Remainder Theorem, $\omega$ is multiplicative. Each interval of $d$
consecutive integers contains exactly $\omega(d)$ solutions of the system (**).\\
Thus 
$$\mathcal{A}_d=\frac{\omega(d)}{d}\:X+R_d,$$
where $|R_d|\leq \omega(d).$\\
Thus, Lemma \ref{x:brun} is applicable and we obtain:
\begin{align*}
|\mathcal{U}_{5,1}|&\leq S(\mathcal{A},\mathcal{P},z)\\
&\ll y\prod_{p\leq y}\left(1-\frac{1}{p}\right)\prod_{\log x<p\leq z}\left(1-\frac{2}{p}\right)\left(1-\frac{1}{p}\right)^{-1}\:.
\end{align*}
Well known estimates of elementary prime number theory as in the proof of Lemma
2.8 in \cite{konyagin}, give the result of Lemma \ref{x:lem44}.
\end{proof}
For the next definitions and results we follow the paper \cite{konyagin}.
\begin{definition}
Let
\begin{eqnarray}
\tilde{\mathcal{P}}_3=\left\{ 
  \begin{array}{l l}
   \left\{p\::\: \frac{x}{40k}<p\leq x,\ p\equiv 2\:(\bmod\:3) \right\} \:, & \quad \text{if $k$ is odd}\vspace{2mm}\\ 
   \left\{p\::\: \frac{x}{40k}<p\leq \frac{x}{2},\ p\equiv 3\:(\bmod\:2k) \right\} \:, & \quad \text{if $k$ is even}\:,\\
  \end{array} \right.
\nonumber
\end{eqnarray}
We now define the exceptional set $\mathcal{U}_6$ as follows:\\
For $k$ odd we set
$$\mathcal{U}_6=\emptyset\:.$$
For $k$ even and $\delta>0$, we set
$$\mathcal{U}_6=\left\{ u\in[-y,y]\::\: \left(\frac{-u}{p}\right)=1\ \ \text{for at most}\ \frac{\delta x}{\log x}\ \text{primes}\ p\in\tilde{\mathcal{P}}_3 \right\}\:.$$
\end{definition}
We shall make use of the following result from \cite{konyagin}.
\begin{lemma}
$$|\mathcal{U}_6| \ll_{\varepsilon} x^{1/2+2\varepsilon}\:.$$
\end{lemma}
\begin{proof}
This is formula (4) from \cite{konyagin}, where $\mathcal{U}_6$ is denoted by $\mathcal{U'}$.
\end{proof}
\begin{definition}\label{x:def47}
We set
$$\mathcal{U}_7=\mathcal{U}_4\cup\mathcal{U}_5\:.$$
\end{definition}
\begin{lemma}\label{x:lem48}
We have 
$$|\mathcal{U}_7|\leq \frac{x}{20k\log x}\:.$$
\end{lemma}
\begin{proof}
This follows from Definition \ref{x:def47} and Lemmas \ref{x:lem43}, \ref{x:lem44}
\end{proof}
We now introduce the congruence conditions, which determine the integer  $m_0$ uniquely $(\bmod\: P(x))$.
\begin{definition}\label{x:def49}
\[
m_0\equiv 1\:(\bmod\: p),\ \text{for}\ p\in\mathcal{P}_1\:, \tag{$C_1$}
\]
\[
m_0\equiv 2\:(\bmod\: p),\ \text{for}\ p\in\mathcal{P}_2\:. \tag{$C_2$}
\]
\end{definition}

\noindent For the introduction of the congruence conditions $(C_3)$ we make use of Lemma 
\ref{x:lem48}.\\
Since 
$$|\tilde{\mathcal{P}}_3|\geq |\mathcal{U}_7|,$$
there is an injective mapping
$$\Phi\::\: \mathcal{U}_4\rightarrow \tilde{\mathcal{P}}_3,\ \ u\rightarrow \mathcal{P}_u.$$
We set $$\mathcal{P}_3=\Phi(\mathcal{U}_4)\:.$$
For all $u$, for which the congruence
$$m^k\equiv -(u-1)\:(\bmod\: p_u)$$
is solvable, choose a solution $m_u$ of this congruence.\\
The set $(C_3)$ of congruences is then defined by 
\[
m_0\equiv m_u\:(\bmod\: p_u),\ \ p_u\in\mathcal{P}_3\:.  \tag{$C_3$}
\]
Let 
$$\mathcal{P}_4=\{p\in[0,x)\::\: p\not\in\mathcal{P}_1\cup\mathcal{P}_2\cup\mathcal{P}_3  \}\:.$$
The set of congruences is then defined by
\[
m_0\equiv 1\:(\bmod\: p),\ \ p\in\mathcal{P}_4\:.  \tag{$C_4$}
\]
\begin{lemma}
The congruence systems $(C_1)-(C_4)$ and the condition $1\leq m_0\leq P(x)-1$
determine $m_0$ uniquely. We have $(m_0, P(x))=1$.
\end{lemma}
\begin{proof}
The uniqueness follows from the Chinese Remainder Theorem. The coprimality 
follows, since by the definition of $(C_1)-(C_4)$ $m_0$ is coprime to all $p$, with
$0< p\leq x$.
\end{proof}
\begin{lemma}
Let $m\equiv m_0\:(\bmod\: P(x))$. Then $(m, P(x))=1$ and the number
$$m^k+(u-1)$$
is composite for all $u\in[-y,y]\setminus \mathcal{U}_6$\:.
\end{lemma}
\begin{proof}
For $u\in\mathcal{U}_1$, there is $p\in\mathcal{P}_1$ with $p\mid u$. Therefore,
since by Definition \ref{x:def49}, the system $(C_1)$ implies that $m_0\equiv 1\:(\bmod\: p)$, we have
$$m^k+(u-1)\equiv m_0^k+(u-1)\equiv 1+u-1\equiv u\equiv 0\:(\bmod\: p)\:,$$
i.e.
$$p\mid m^k+(u-1)\:.$$
For $u\in\mathcal{U}_3$, $u\not\in\mathcal{U}_5$, there is $p\in\mathcal{P}_2$
with $p\mid  u+2^k-1.$\\
Since by $(C_2)$ $m_0\equiv2\:(\bmod\: p)$, we have
$$m_0^k+(u-1)\equiv 2^k-2^k\equiv0\:(\bmod\: p)\:,$$
i.e.
$$p\mid m^k+(u-1)\:.$$
The remaining cases, except $u\in\mathcal{U}_6$, are checked similarly.
\end{proof}
\section{Conclusion of the proof of Theorem \ref{main}}
Let now $x$ be such that $P(x)$ is a good modulus in the sense of Definition
\ref{x:definition31}. By Lemma \ref{x:lem32}, there are arbitrarily large such elements $x$.
Let $D$ be a sufficiently large positive integer.
Let $\mathcal{M}$ be the matrix with $P(x)^{D-1}$ rows and $U=2\lfloor y\rfloor+1$
columns, with the $r,u$ element being
$$a_{r,u}=(m+rP(x))^k+u-1,$$
where $1\leq r\leq P(x)^{D-1}$ and $-y\leq u\leq y$.\\
Let $N(x,k)$ be the number of perfect $k$-th powers of primes in the column
$$\mathcal{C}_1=\{a_{r,1}\::\: 1\leq r\leq P(x)^{D-1}  \}\:.$$
Since $P(x)$ is a good modulus, we have by Lemma \ref{x:lem32} that
\[
N_0(x,k)\geq C_0(k)\:\frac{P(x)^{D-1}}{\log(P(x)^{D-1})}\:.\tag{5.1}
\]
Let $\mathcal{R}_1$ be the set of rows ${R}_1$, in which these primes appear. We now give an upper bound for the number $N_1$ of rows ${R}_r\in\mathcal{R}_1$, which contain primes.\\
We observe that for all other rows  ${R}_r\in\mathcal{R}_1$, the element 
$$a_{r,1}=(m_0+rP(x))^k$$
is a prime avoiding $k$-th power of the prime $m_0+rP(x)$.
\begin{lemma}\label{x:lem5151}
For sufficiently small $c_2$, we have
$$N_1\leq \frac{1}{2}N_0(x,k)\:.$$
\end{lemma}
\begin{proof}
For all $v$ with $v-1\in\mathcal{U}_6$, let 
$$T(v)=\{ r\::\: 1\leq r\leq P(x)^{D-1},\ m_0+rP(x)\ \text{and}\ (m_0+rP(x))^k+v-1
\ \text{are primes} \}.$$
We have
\[
N_1\leq \sum_{v\in\mathcal{U}_6}T(v)\:.\tag{5.2}
\]
For the estimate of $T(v)$ we apply again Lemma \ref{x:brun}.\\ We set
\begin{align*}
g(r)&=m_0+rP(x)\\
h(r)&=(m_0+rP(x))^k+v-1\\
\mathcal{A}&=\{ g(r)h(r)\::\: 1\leq r \leq P(x)^{D-1}\}\:,\\
\mathcal{A}_p&=\{n\in\mathcal{A}\::\: n\equiv 0\:(\bmod\: p)\}, \text{for each prime}\
p\ \text{with}\ x<p\leq P(x)\:.
\end{align*}
We let $\omega(p)$ be the number of solutions of the congruence 
$$g(r)h(r)\equiv 0\:(\bmod\: p),\ \ \text{for}\ \ p>x\:.$$
Since $p\nmid P(x)$, the linear congruence 
$$g(r)\equiv 0\:(\bmod\: p)$$
has exactly one solution.\\
Let
$$\rho(p)=\left| \{n\:(\bmod\: p)\::\: n^k+v-1\equiv\:0\:(\bmod\:p)\} \right|\:. $$
Then the congruence
$$h(r)\equiv0\:(\bmod\: p)$$
has $\rho(p)$ solutions $(\bmod\: p)$.\\
By Lemma \ref{x:brun}, we have:
\begin{align*}
T(v)&\leq S(\mathcal{A},\mathcal{P},P(x))\\
&\ll P(x)^{D-1}\:\prod_{x<p\leq P(x)}\left(1-\frac{1}{p}\right)\prod_{x<p\leq P(x)}\left(1-\frac{\rho(p)}{p}\right)\:.\tag{5.3}
\end{align*}
By Lemma 3.1 of \cite{konyagin}, we have
$$\prod_{x<p\leq P(x)}\left(1-\frac{\rho(p)}{p}\right)\ll_{k,\varepsilon} |v|^\varepsilon\:
\frac{\log x}{\log P(x)}\:.$$
Lemma \ref{x:lem5151} now follows from (5.2), (5.3) and the bound for $|\mathcal{U}_6|$.\\
This completes the proof of Theorem \ref{main}.
\end{proof}

\end{document}